\documentclass{amsart}

\usepackage{amssymb}
\usepackage{color}
\usepackage{amsxtra} 
\usepackage{mathrsfs} 
\usepackage{yfonts}

\newtheorem{theorem}{Theorem}
\newtheorem{lem}[theorem]{Lemma}
\newtheorem{prop}[theorem]{Proposition}
 
\newtheorem*{theorem*}{Main Theorem} 
\newtheorem*{theorem**}{Theorem}

\theoremstyle{definition}
\newtheorem{definition}[theorem]{Definition}

\theoremstyle{remark}
\newtheorem{remark}[theorem]{Remark} 
\newtheorem*{remark*}{Remark}

\numberwithin{equation}{section}
\numberwithin{theorem}{section} 

\begin{document}

\title[A geometric instability of the Euler flows  
]
{A geometric instability of the  laminar axisymmetric Euler flows
with  oscillating 
flux
}


\author{Tsuyoshi Yoneda}
\address{Graduate School of Mathematical Sciences, University of Tokyo, Komaba 3-8-1 Meguro, Tokyo 153-8914, Japan} 
\email{yoneda@ms.u-tokyo.ac.jp}


\date{\today} 



\begin{abstract} 
The dynamics along the particle trajectories for the 3D axisymmetric Euler equations in an infinite cylinder are considered. It is shown that if the inflow-outflow is highly oscillating in time,
the corresponding Euler flow cannot keep the uniformly smooth laminar profile provided that 
the swirling component  is not zero. In the proof, Frenet-Serret formulas and orthonormal moving frame are essentially used.
\end{abstract} 

\maketitle

\section{Introduction} 
\label{sec:Intro} 
We study the dynamics along  the particle trajectories for the 3D axisymmetric Euler equations.
Such  Lagrangian dynamics have already been studied in mathematics (see \cite{C0, C1,C2}). For example, in \cite{C1}, Chae considered a blow-up problem for the axisymmetric 3D incompressible Euler equations with swirl. 
More precisely, he showed that under some assumption of local minima for the pressure on the axis of symmetry
with respect to the radial variations along some particle trajectory, the solution blows up in finite time.
Although the blowup problem of 3D Euler equation is still an outstanding open problem, in this paper, we focus on a different problem in physics, especially,  the cardiovascular system \cite{FQV}. 
If the blood flow is in large and medium sized vessels, the flow is governed by the usual incompressible Navier-Stokes equations. 
In this paper we focus on behavior of the interior flow, thus it is reasonable to use 
a simpler model: the 3D axisymmetric Euler flow in an infinite cylinder
$\Omega:=\{x\in\mathbb{R}^3: \sqrt{x_1^2+x_2^2}<1,\ x\in\mathbb{R}
\}$.
The configuration of the boundary is not important anymore, thus the setting $\Omega$ is just for simplicity.
The incompressible Euler equations are expresses as follows:
\begin{eqnarray}
\label{Euler eq.}
& &\partial_tu+(u\cdot \nabla)u=-\nabla \pi,\quad\nabla \cdot u=0\quad \text{in}\quad \Omega,\\
\nonumber
& & \quad u|_{t=0}=u_0,\quad u\cdot n=0 \quad \text{on}\quad \partial\Omega,\quad
u(x,t)\to (0,0,g(t))\quad (x_3\to\pm\infty)
\end{eqnarray}
with $u=u(x,t)=(u_1(x_1,x_2,x_3,t),u_2(x_1,x_2,x_3,t),u_3(x_1,x_2,x_3,t))$, 
 $\pi=\pi(x,t)$ and an uniform inflow-outflow condition  $g=g(t)$ (the uniform setting is just for simplicity,
we can easily generalize it), where $n$ is a unit normal vector on the boundary.
We show a geometric instability of the laminar profile (more precisely, ``non-uniformly smooth laminar profile" which will be defined rigorously later) when the uniform inflow-outflow is 
highly oscillating in time, more precisely,
\begin{equation}\label{inflow-outflow condition}
g(t)=2+(1-t)^{\beta_1}\sin \left((1-t)^{-\beta_2}\right)\quad (\beta_1, \beta_2>0).
\end{equation}
Note that $u=(0,0,g)$ in $\Omega\times [0,1)$ is one of the solution to \eqref{Euler eq.}.
Throughout this paper we assume existence of a unique smooth solution to \eqref{Euler eq.} in $t\in [0,1)$.
If there is no unique smooth solution in $t\in [0,1)$ for some initial data, then we can regard it as one of the instability.
The above inflow-outflow has the following property:
for any $\epsilon>0$, there is
a time interval $I$ such that    
\begin{equation}\label{flux condition}
1<\frac{\epsilon g'(t)^2}{g(t)}<\epsilon^2 g''(t)\quad\text{for}\quad
t\in I\subset [1-\epsilon,1).
\end{equation}
Throughout this paper we always focus on the flow behavior in $I$ to extract a dominant term.
In this case, notations ``$\approx$" and ``$\lesssim$" are convenient. The notation ``$a \approx b$" means there is a positive constant $C>0$ such that 
\begin{equation*}
C^{-1}a\leq b\leq Cb,
\end{equation*}
and ``$a\lesssim 1$" means that there is a positive constant $C>0$ such that 
\begin{equation*}
0\leq a\leq C.
\end{equation*}
\begin{remark}
If $g(t)=(1-t)^{-\beta}$ ($\beta>0$), then it does not satisfy \eqref{flux condition}.
Thus the oscillating flux setting might be essential.
\end{remark}

This flow setting arises from a reduced cardiovascular 1D model \cite[Section 10]{FQV}.
To obtain the reduced model, we need to assume the flow is always unilateral laminar flow,
especially,
for $D:=\{(x_1,x_2)\in\mathbb{R}^2: \sqrt{x_1^2+x_2^2}<1\}$, the axis direction of the flow $u_3$ is assumed to satisfy
\begin{equation}\label{unilateral flow condition}
\int_{D}u_3(x_1,x_2,x_3,t)^2dx_1dx_2=\alpha \left(\int_Du_3(x_1,x_2,x_3,t)dx_1dx_2\right)^2
\end{equation}
for some positive constant $\alpha>0$ (see \cite[(10.18)]{FQV}).
However, in this setting, it is not clear whether or not such  condition \eqref{unilateral flow condition} is always valid. For example, if the flow is not unilateral, containing the  reverse flow, 
then $\alpha$ may become infinity.
In this paper 
we 
 show all axisymmetric Euler flows with swirl component have  non-uniformly smooth laminar profile (possibly, turbulent transition) when these  corresponding inflow-outflow are highly oscillating in time.
These non-uniformly smooth laminar profiles suggest us that we may need to construct more suitable cardiovascular 1D model with which such non-uniformly smooth laminar profiles are more involved.

Since we consider the axisymmetric Euler flow, we can simplify the Euler equations \eqref{Euler eq.}.
Let $e_r:= x_h/|x_h|$,
$e_\theta:=x_h^\perp/|x_h|$ and
  $e_z=(0,0,1)$ with
 $x_h=(x_1,x_2,0)$, $x_h^\perp=(-x_2,x_1,0)$.
The vector valued function $u$ can be rewritten as   $u=v_re_r+v_\theta e_\theta+v_ze_z$,  
where $v_r=v_r(r,z,t)$, $v_\theta=v_\theta(r,z,t)$ and $v_z=v_z(r,z,t)$
with $r=|x_h|$, $z=x_3$.
Then the   axisymmetric Euler equations can be expressed  
as follows: 
\begin{eqnarray}
\partial_t v_r+v_r\partial_rv_r+v_z\partial_zv_r-\frac{v_\theta^2}{r}+\partial_r p&=&
0,\\
\label{axisymmetricEuler-1}
\partial_tv_\theta+v_r\partial_rv_\theta+v_z\partial_zv_\theta+\frac{v_rv_\theta}{r}&=&
0,\\
\partial_tv_z+v_r\partial_rv_z+v_z\partial_zv_z+\partial_zp&=&0
,\\
\label{axisymmetricEuler-2}
\frac{\partial_r(rv_r)}{r}+\partial_zv_z&=&0.
\end{eqnarray}

In order to state ``uniformly smooth laminar profiles" rigorously,  
first we need to give
several definitions.
\begin{definition}
We call ``unilateral flow" iff $v_z=u\cdot e_z>0$ in $\Omega$.
\end{definition}
\begin{definition}\label{Stream-shell near the boundary} 
(Axis-length  streamline in $z$.)\ 
For a unilateral flow,
 we can define an axis-length streamline $\gamma(z)$.
Let $t$
 be fixed, and let $\gamma(z)$ be such that 
\begin{equation*}
\gamma(\bar r_0,z,t)=\gamma (z):=(\bar R(z)\cos \bar \Theta(z), \bar  R(z)\sin \bar \Theta(z), z)
\end{equation*} 
with
$\bar R(z)=\bar R(\bar r_0,z,t)$, $\bar R(\bar r_0,0,t)=\bar r_0$, $\bar \Theta(z)=\bar \Theta(z,t)$
 and we choose $\bar R$ and $\bar \Theta$ in order to satisfy 
\begin{equation*}
\partial_z\gamma(z)=\left(\frac{u}{u\cdot e_z}\right)(\gamma(z),t).
\end{equation*}
\end{definition}
We easily see 
\begin{equation*}
\partial_z\gamma\cdot e_z=1,\quad \partial_z\gamma\cdot e_r=\partial_z\bar R=\frac{u_r}{u_z}\quad\text{and}\quad
\partial_z \gamma\cdot e_\theta=\bar R\partial_z\Theta=\frac{u_\theta}{u_z}.
\end{equation*}
Since $\partial_{\bar r_0}\bar R>0$ (otherwise uniqueness does not hold), 
we have 
its inverse $r_0=\bar R^{-1}(r,z,t)$.
In order to define  ``uniformly smooth laminar profiles", we use 
the regularity of $\bar R$ and $\bar R^{-1}$ up to three derivatives.
Note that 
the definition of  ``laminar profile" should come from  geometry, thus it seems $\bar R$ and $\bar R^{-1}$ are the suitable concepts rather than the velocity.

\begin{definition} (uniformly smooth laminar profile.)\ 
Let $\partial=\partial_{z}$ or $\partial_{\bar r_0}$, and let $\bar \partial=\partial_{z}$ or $\partial_r$. We call `` uniformly smooth laminar profile" if and only if 
 $\bar R$ and $\bar R^{-1}$ satisfy the following
\begin{equation*}
 \partial_{\bar r_0}\bar R\approx 1,\quad
|\partial^\ell\bar R|,\ 
 |\bar \partial^\ell\bar R^{-1}|,\ 
|\partial_t\bar R^{-1}|,\ 
|\partial_t\partial_{\bar r_0}\bar R|
\lesssim 1
\end{equation*}
for $t\in[0,1)$,
 $\ell=1,2,3$.
Later we deal with the curvature and torsion of the particle trajectory, thus it is natural to see up to three derivatives.
\end{definition}

\begin{remark}
As we remarked that $u=(0,0,g)$ in $\Omega\times[0,1)$ is one of the solution to \eqref{Euler eq.}.
This flow is the typical laminar flow.
In this case 
\begin{equation*}
 \partial_{\bar r_0}\bar R=\partial_{r}\bar R^{-1}= 1,\quad
\partial^\ell\partial_z\bar R= 
 \bar \partial^\ell\partial_z\bar R^{-1}
= 
\partial_t\bar R^{-1}=
\partial_t\partial_{\bar r_0}\bar R=0
\end{equation*}
for $t\in[0,1)$,
 $\ell=0,1,2$.
\end{remark}



Now we define the particle trajectory.
The associated Lagrangian flow $\eta(t)$
is a solution of the initial value problem 
\begin{align} \label{eq:flowC} 
&\frac{d}{dt}\eta(x,t) = u(\eta(x,t),t), 
\\  \label{eq:flow-icC} 
&\eta(x,0) = x. 
\end{align}



Now we give the main theorem.

\begin{theorem}
Let $X(0)$ be such that 
\begin{equation*}
X(0):=\{x\in\Omega: u_0(x)\cdot e_\theta\not=0 \}
\end{equation*}
and $X(t)$ be 
\begin{equation*}
X(t):=\{\eta(x,t)\in\Omega: x\in X(0)\}.
\end{equation*}
Assume there is a unique  smooth solution to the Euler equations \eqref{Euler eq.} in $\Omega\times [0,1)$.
For any $x\in \Omega(0)$, then at least, either of the following two cases must happen:

\begin{itemize}

\item
Its corresponding laminar profile at $\eta(x,t)\in \Omega(t)$ is not uniformly smooth,

\item
The particle  $\eta(x,t)\in \Omega(t)$ touches the axis in $t\in (0,1]$.

\end{itemize}
\end{theorem}

In the next section, we prove the main theorem.

\section{Proof of the main theorem.}
In order to prove the main theorem, we define the Lagrangian flow along $r$,$z$-direction. Let 
\begin{eqnarray}\label{2D-trajectory-1}
& &\frac{d}{dt}Z(t)=v_z(R(t),Z(t),t),\\
\nonumber
& &Z(0)=z_0
\end{eqnarray}
and
\begin{eqnarray}\label{2D-trajectory-2}
& &\frac{d}{dt}R(t)=v_r(R(t),Z(t),t),\\
\nonumber
& &R(0)=r_0
\end{eqnarray}
with $Z(t)=Z(r_0,z_0,t)$ and $R(t)=R(r_0,z_0,t)$.
Assume the axisymmetric smooth Euler flow has ``uniformly smooth laminar profile" and ``particles never touch the axis".
Rigorously, ``particles never touch the axis"  means that 
$R$ satisfies the following: 
For any $\tilde r_0>0$, there is $C>0$ such that 
\begin{equation*}
R(t)>C\quad\text{for}\quad r_0>\tilde r_0\quad\text{and}\quad t\in[0,1).
\end{equation*}
First we express $v_z$ and $v_r$ by using $\bar R$ and $\bar R^{-1}$. To do so, we define the cross section of the stream-tube (annulus). Let 
$B_{-\infty}(\bar r_0)=\{x\in\mathbb{R}^3: |x_h|<\bar r_0,\ x_3=-\infty\}$ and let
\begin{equation*}
  A(\bar r_0,z,\epsilon,t):=\bigcup_{x\in B_{-\infty}(\bar r_0+\epsilon)\setminus B_{-\infty}(\bar r_0)}
\gamma(x,z,t).
\end{equation*}
We see that  its measure is  
\begin{equation*}
|A(\bar r_0,z,\epsilon,t)|=\pi\left(\bar R(\bar r_0,\epsilon,z,t)^2-\bar R(\bar r_0,z,t)^2\right).
\end{equation*}
\begin{definition} (Inflow propagation.)\ 
Let $\rho$ be such that  

\begin{equation*}
\rho(\bar r_0,z,t):=\lim_{\epsilon\to 0}\frac{|A(\bar r_0,-\infty,\epsilon,t)|}{|A(\bar r_0,z,\epsilon,t)|}.
\end{equation*}
\end{definition}
We see that
\begin{equation*}
\rho(\bar r_0,z,t)=\frac{\partial_{\bar r_0}\bar R(\bar r_0,-\infty,t) \bar R(\bar r_0,-\infty,t)}{\partial_{\bar r_0}\bar R(\bar r_0,z,t) \bar R(\bar r_0,z,t)}=\frac{\bar r_0}{\partial_{\bar r_0}\bar R(\bar r_0,z,t) \bar R(\bar r_0,z,t)}
=\frac{2\bar r_0}{\partial_{\bar r_0}\bar R(\bar r_0,z,t)^2}.
\end{equation*}
\begin{remark}
If the laminar profile is uniformly smooth, then we have  the 
estimates of the inflow propagation $\rho$:
\begin{equation*}
\rho\approx 1
,\  
|\partial_z\rho|, |\partial_z^2\rho|,
|\partial_{\bar r_0}\rho|, |\partial_{\bar r_0}^2\rho|
\lesssim 1
\quad\text{for}\quad t\in[0,1).
\end{equation*}
\end{remark}

Since 
\begin{equation*}
2\pi \int_{\bar R(\bar r_0,z,t)}^{\bar R(\bar r_0+\epsilon,z,t)}u_z(r',z,t)r'dr'=2\pi\int_{\bar r_0}^{\bar r_0+\epsilon}u_z(r',-\infty,t)r'dr'
\end{equation*}
 by divergence free and Gauss's divergence theorem,
 we can figure out $v_z$ by using the inflow propagation $\rho$, 
\begin{eqnarray*}
v_z(r,z,t)&=&\lim_{\epsilon\to 0}\frac{2\pi}{|A(\bar r_0,z,\epsilon,t)|}\int_{\bar R(\bar r_0,z,t)}^{\bar R(\bar r_0+\epsilon,z,t)}v_z(r',z,t)r'dr'\\
&=&
\lim_{\epsilon\to 0}\frac{|A(\bar r_0,-\infty,\epsilon,t)|}{|A(\bar r_0,z,\epsilon,t)|}\frac{2\pi}{|A(\bar r_0,-\infty,\epsilon,t)|}\int_{\bar r_0}^{\bar r_0+\epsilon}v_z(r',-\infty,t)r'dr'\\
&=&
\rho(\bar r_0,z,t)u_z(\bar r_0,-\infty,t).
\end{eqnarray*}
Thus we have the following proposition.

\begin{prop}\label{formula of velocities}
We have the following formula of $v_z$ and $v_r$:
\begin{equation}\label{!}
v_z(r,z,t)=\rho(\bar R^{-1}(r,z,t),z,t)u_z(\bar R^{-1}(r,z,t),0,t)=\rho(\bar R^{-1},z,t)g(t)
\end{equation}
and 
\begin{equation}\label{!!}
v_r(r,z,t)=(\partial_z\bar R)(\bar R^{-1}(r,z,t), z,t)u_z(r,z,t).
\end{equation}
\end{prop}
By the above proposition, we see 
\begin{equation*}
|v_z|,\ |v_r|\lesssim 1.
\end{equation*}
Since $v_z>0$, then we can define the inverse of $Z$ in $t$: $t=Z_t^{-1}(z,r_0,z_0)$.
In this case we can estimate $\partial_zZ^{-1}_t=1/\partial_tZ=1/v_z\approx 1/g(t)$ and $\partial_z^2Z^{-1}_t=-\frac{\partial_zv_z}{v_z^2}$.
First we show the following estimates.
\begin{lem}\label{estimates of v}
For $t\in I$,
we have the following estimates along the axis-length trajectory:
\begin{eqnarray}\label{partial_zv_z}
& &\partial_zv_z(R(Z^{-1}_t(z)),z,Z_t^{-1}(z))\approx g'(t)/g(t),\\
\nonumber
& &\partial_z^2Z_t^{-1}\approx- g'(t)/g(t)^3,\\
\nonumber
& &\partial_z^2v_z(R(Z^{-1}_t(z)),z,Z_t^{-1}(z))\approx g''(t)/g(t)^2,\\
\nonumber
& &|\partial_zv_r(R(Z^{-1}_t(z)),z,Z_t^{-1}(z))|\lesssim g'(t)/g(t),\\
\nonumber
& &|\partial_z^2v_r(R(Z^{-1}_t(z)),z,Z_t^{-1}(z))|\lesssim g''(t)/g(t)^2\quad\text{with}\quad
t=Z^{-1}_t(z).
\end{eqnarray} 
Moreover, we have
\begin{equation}\label{v_theta estimate}
v_\theta(R(Z^{-1}_t(z)),z,Z^{-1}_t(z))\approx 1,
\end{equation} 
(it is reasonable to assume $v_\theta(r_0,z_0,0)>0$)
\begin{eqnarray}\label{partial_zv_theta}
& &|\partial_zv_\theta(R(Z^{-1}_t(z)),z,Z_t^{-1}(z))|\lesssim 1,\\
\nonumber
& &|\partial_z^2v_\theta(R(Z^{-1}_t(z)),z,Z_t^{-1}(z))|\lesssim g'(t)/g(t)\quad\text{with}\quad
t=Z^{-1}_t(z)
\end{eqnarray}
and
\begin{equation*}
\partial_t|u(\eta(x,t),t)|\approx g'(t).
\end{equation*}
\end{lem}

\begin{proof}
First we consider a non-incompressible 2D-flow composed by $R$ and $Z$.
Let us denote $\eta_{2D}=\eta_{2D}(t)=(R(t),Z(t))$ and $D\eta_{2D}$ be its Lagrangian deformation:
\begin{equation*}
D\eta_{2D}=
\begin{pmatrix}
\partial_{r_0}R&\partial_{z_0}R\\
\partial_{r_0}Z&\partial_{z_0}Z
\end{pmatrix}.
\end{equation*}
We see $\det (D\eta_{2D})=\partial_{r_0}R\partial_{z_0}Z-\partial_{z_0}R\partial_{r_0}Z$ and thus we have 
\begin{equation*}
D(\eta_{2D}^{-1})=(D\eta_{2D})^{-1}=\frac{1}{\det D\eta_{2D}}
\begin{pmatrix}
\partial_{z_0}Z&-\partial_{z_0}R\\
-\partial_{r_0}Z&\partial_{r_0}R
\end{pmatrix}.
\end{equation*}
A direct calculation with \eqref{axisymmetricEuler-2}, \eqref{2D-trajectory-1} and \eqref{2D-trajectory-2} yields
\begin{equation*}
\frac{d}{dt}(\det D\eta_{2D})=(\partial_rv_r+\partial_zv_z)(\det D\eta_{2D})=-\frac{v_r}{R(t)}(\det D\eta_{2D}).
\end{equation*}
Thus
\begin{equation*}
\det D\eta_{2D}(t)=\det D\eta_{2D}(0)\exp\bigg\{-\int_0^t\frac{v_r(R(\tau),Z(\tau),\tau)}{R(\tau)}d\tau\bigg\}.
\end{equation*}
Since $|v_r|\lesssim 1$ and $R(t)\geq C$,
we have  
\begin{equation*}
\det D\eta_{2D}\approx 1
\end{equation*}
by Gronwall's equality.
Recall the inflow propagation  $v_z(\bar R^{-1}(R(t),Z(t),t),-\infty,t)=g(t)$.
Then we see
\begin{equation}\label{explicit formula1}
v_z(R(t),Z(t),t)=\rho(\bar R^{-1}(R(t), Z(t),t), Z(t),t)g(t)=\partial_tZ(t)
\end{equation}
and 
\begin{equation}\label{explicit formula2}
(\partial_z\bar R)(\bar R^{-1}(R,Z,t),Z,t)\rho(\bar R^{-1}(R,Z,t),Z,t)g(t)=v_r(R,Z,t)=\partial_tR.
\end{equation}
Since we have already controlled $\det D\eta_{2D}$,
here we  estimate $\partial_{r_0}R$, $\partial_{r_0}Z$, $\partial_{z_0}R$ and $\partial_{z_0}Z$ respectively.
We see the following estimates of Lagrangian deformation:
\begin{equation*}
\partial_t\partial_{z_0}Z(t)=\bigg[\partial_{z_0}R\partial_{\bar r_0}\rho\partial_r\bar R^{-1}+
\partial_{z_0}Z\partial_{\bar r_0}\rho\partial_z\bar R^{-1}+\partial_{z_0}Z\partial_z\rho\bigg]g(t)\\
\end{equation*}
\begin{equation*}
\partial_t\partial_{z_0}R(t)=\bigg[\partial_{z_0}\partial_z\bar R\partial_{r}\bar R^{-1}\partial_{z_0}R+
\partial_{r_0}\partial_z\bar R\partial_z\bar R^{-1}\partial_{z_0}Z+\partial_z^2\bar R\partial_{z_0}Z\bigg]g(t)+(v_z\text{ part}).
\end{equation*}
Then we can construct a Gronwall's inequality of $|\partial_{z_0}Z|+|\partial_{z_0}R|$
and then we can control each $|\partial_{z_0}Z|$ and $|\partial_{z_0}R|$ (here we use ``uniformly smooth laminar profile"). 
By the same calculation, we can also control
$|\partial_{r_0}Z|$ and $|\partial_{r_0}R|$.
This means 
\begin{equation}\label{control streamline1}
|\partial_{z_0}Z|,\ |\partial_{z_0}R|,\  
|\partial_{r_0}Z|,\ |\partial_{r_0}R|
\lesssim 1.
\end{equation}
By the same argument again,  we have 
\begin{equation}\label{control streamline2}
|\partial^2 Z|,\ |\partial^2 R|\lesssim 1
\end{equation}
with $\partial=\partial_{r_0}$ or $\partial_{z_0}$.
By \eqref{control streamline1} and \eqref{control streamline2},  we can estimate $v_r$ and $v_z$ along the particle trajectory in $z$-valuable (note that the dominant terms are always $g$ with derivatives).
Thus we obtain \eqref{partial_zv_z}.
Now we control $v_\theta$ by using \eqref{partial_zv_z}.
By \eqref{axisymmetricEuler-1} we see that 
\begin{equation*}
\partial_t v_\theta(R(t),Z(t),t)=v_\theta(r_0,z_0,0)-\frac{v_r(R(t),Z(t),t)v_\theta(R(t),Z(t),t)}{R(t)}.
\end{equation*}
Applying the Gronwall equality,  we see
\begin{equation}\label{v_theta}
v_\theta(R^{-1}(Z^{-1}_t(z)),z,Z^{-1}_t(z))=v_\theta(r_0,z_0,0)\exp\bigg\{-\int_0^{Z_t^{-1}(z)}
\frac{v_r(R(\tau),Z(\tau),\tau)}{R(\tau)}d\tau\bigg\}.
\end{equation}
Since $|v_r|\lesssim 1$, we have \eqref{v_theta estimate}.
Just taking derivatives to \eqref{v_theta} in $z$-valuable, then we also have \eqref{partial_zv_theta}.
Now we estimate $\partial_t|u(\eta,t)|$.
Recall  the usual trajectory $\eta(x,t)=(R(t)\cos\Theta(t),R(t)\sin\Theta(t),Z(t))$,\\ $e_\theta=(-\sin\Theta(t),\cos\Theta(t),0)$ and 
$e_r=(\cos\Theta(t),\sin\Theta(t),0)$. Then, by a direct calculation with $u=v_re_r+v_\theta e_\theta+v_ze_z$,  we see that 
\begin{eqnarray*}
& &\frac{1}{2}\partial_t|u(\eta(x,t),t)|^2=\partial_t u\cdot u
=\partial_t v_rv_r+\partial_tv_\theta v_\theta+\partial_tv_zv_z
\end{eqnarray*}
along the trajectory.
Just take a time derivative to $v_z$ along the trajectory, then we have  
\begin{eqnarray*}
\partial_t(v_z(R(t),Z(t),t))&=&\partial_{\bar r_0}\rho\partial_t\bar R^{-1}g(t)+\partial_{\bar r_0}\rho\partial_z\bar R^{-1}\partial_t Z g
+
\partial_{\bar r_0}\rho\partial_r\bar R^{-1}\partial_tR g\\
& &
+\partial_z\rho\partial_t Z\rho
+\partial_t\rho g
+\rho g'.\\
\end{eqnarray*}
Thus
\begin{equation*}
\partial_t v_z v_z\approx\rho^2g'(t)g(t)\quad\text{for}\quad t\in I.
\end{equation*}
By the similar calculation,
\begin{equation*}
\partial_t v_r v_r\lesssim
(\partial_z\bar R)^2\rho^2g'(t)g(t)\quad\text{for}\quad t\in I.
\end{equation*}
Clearly $\partial_tv_\theta v_\theta$ is
 not large anymore.
Thus we have 
\begin{equation*}
\partial_t|u(\eta(x,t),t)|^2\approx g'(t)g(t)\quad\text{for}\quad t\in I
\end{equation*}
and then
\begin{equation*}
\partial_t|u(\eta(x,t),t)|\approx g'(t)\quad\text{for}\quad t\in I.
\end{equation*}
\end{proof}

Now we define the axis-length trajectory $\tilde\eta$ in $z$.

\begin{definition} (Axis-length trajectory.)\ 
Let $\tilde\eta$ be such that 
\begin{equation*}
\tilde \eta(z):=(r(z)\cos\theta(z),r(z)\sin\theta(z),z)
\end{equation*}
and we choose $r(z)$ and $\theta(z)$
in order to satisfy $\tilde \eta(z)=\eta(x,Z^{-1}_t(z))$.
\end{definition}
For $t\in I$, we see
\begin{eqnarray*}
\partial_z\tilde\eta\cdot e_\theta&=&\frac{\partial_t\eta\cdot e_\theta}
{v_z}=r\theta'=\frac{v_\theta(R(Z^{-1}_t(z)),z,Z^{-1}_t(z))}{v_z(R(Z^{-1}_t(z)),z,Z^{-1}_t(z))}\approx 1/g(t),\\
\partial_z\tilde\eta\cdot e_r&=&\frac{v_r}{v_z}=r',\quad |r'|\lesssim 1,\quad |r''|\lesssim g'/g
\end{eqnarray*}
with $t=Z^{-1}_t(z)$. We need the estimates $\theta''$ and $\theta'''$ to specify the geometry of the trajectory, in particular, its 
curvature and torsion.
By Lemma \ref{estimates of v}, we can immediately obtain the following proposition.
\begin{prop}\label{estimates of theta}
For $t\in I$, by Lemma \ref{estimates of v}
 we have (just see the highest order term)
\begin{equation*}
\theta''(z)\approx -\frac{v_\theta\partial_zv_z}{v_z^2}\approx  -g'(t)/g(t)^3
\end{equation*}
 and 
\begin{equation*}
\theta'''(z)\approx -\frac{v_\theta\partial_z^2v_z}{v_z^2}+\frac{2v_\theta(\partial_zv_z)^2}{v_z^3}\approx -g''(t)/g(t)^4
\end{equation*}
with $t=Z^{-1}_t(z)$.
\end{prop}

From  the trajectory $\eta(x,t)$, we define the arc-length trajectory $\eta^*(s)=\eta^*(x,s)$.

\begin{definition} (Arc-length trajectory.)\ 
Let $\eta^*$ be such that 
\begin{equation*}
\eta^*(s):=\eta(x,t(s))\quad \text{and}\quad \eta^*(x,0)=\eta(x,0)
\end{equation*}
with  $\partial_st(s)=|u|^{-1}$.
\end{definition}
In this case we see $|\partial_s\eta^*(s)|=1$.
We define the unit tangent vector $\tau$ as 
\begin{equation*}
\tau(s)=\partial_s \eta^*(x,s),
\end{equation*}
the unit curvature vector $n$ as $\kappa n=\partial_s \tau$ with a curvature function $\kappa(s)>0$,
the unit torsion vector $b$ 
as : $b(s):=\pm\tau(s)\times n(s)$ ($\times$ is an exterior product)
 with a torsion function to be positive $T(s)>0$ (once we restrict $T$ to be positive, then the direction of $b$ can be determined), that is,
\begin{equation*}
Tb:=\partial_sn+\kappa \tau,\quad |b|=1
\end{equation*}
due to the Frenet-Serret formula.

By the estimates of $\theta''$ and $\theta'''$ in  Proposition \ref{estimates of theta},
 we  obtain the following lemma.
\begin{lem}
For $t\in I$, we have 
$n\cdot e_\theta\to -1
$ ($t\to 1$),
$\partial_s\kappa\approx g''(t)/g(t)^4$ (with $t=t(s)$) and $\partial_s\kappa\gg  |\kappa Tb\cdot e_\theta| 
$.
\end{lem}

\begin{proof}
Recall  the arc-length  trajectory ($z=z(s)$): 
\begin{equation*}
\eta^* (x,s)=\tilde\eta(x,z)=(r(z)\cos \theta(z), r(z)\sin \theta(z), z)\quad\text{with}\quad \theta'>0.
\end{equation*}
Thus
$\tau$ and  $\kappa n$ are expressed as 
\begin{equation*}
\tau=(\partial_z\tilde\eta)z',\quad \kappa n=\partial_s^2\eta^*=\partial_z^2\tilde\eta(z')^2+\partial_z\tilde\eta z''.
\end{equation*}
We recall that 
\begin{equation*}
\partial_z\tilde\eta\cdot e_z=1,\quad \partial_z\tilde\eta\cdot e_\theta=r\theta'=\frac{u_\theta}{u_z},\ 
\partial_z\tilde\eta\cdot e_r=r'=\frac{u_r}{u_z}.
\end{equation*}
Clearly, 
$r'=\partial_z\bar R
$ (this essentially links the streamline and the trajectory for fixed $t$).
Near the possible blowup time, we easily see that 
$\theta'$ is small enough and positive.
We also see that 
\begin{eqnarray*}
\partial_z\tilde \eta(x, z)&=&
(-r\theta'\sin\theta,r\theta'\cos\theta,1)+(r'\cos\theta,r'\sin\theta, 0)
\approx (r'\cos\theta,r'\sin\theta,1)
,\\
\partial_z^2\tilde \eta(x,z)&=&
-r(\theta')^2(\cos\theta,\sin\theta,0)+(-r\theta''\sin\theta, r\theta''\cos\theta,0)
\\
& &
+
r''(\cos\theta,\sin\theta,0)+2r'\theta'(-\sin\theta,\cos\theta,0)\\
&=&
r\theta''(-\sin\theta,\cos\theta,0)+r''(\cos\theta,\sin\theta,0)+\text{remainder}
\\
z'(s)&=&(1+(r')^2+(r\theta')^2)^{-1/2}=(1+(r')^2)^{-1/2}+\text{remainder}\\
z''(s)&=&-(1+(r')^2+(r\theta')^2)^{-2}(r'r''+r\theta'(r'\theta'+r\theta''))\\
&=&
-(1+(r')^2)^{-2}(r'r''+r^2\theta'\theta'')+\text{remainder}.
\end{eqnarray*}
  $r'$, $r''$  and $r'''$ can be controlled due to the uniformly smooth laminar profile, and  recall that  $\theta''\approx -g'(t)/g(t)^3$. Thus
\begin{eqnarray*}
\kappa^2&=&|\kappa n|^2
=|\partial_z^2\eta|^2(z')^4+
2(\partial_z\eta\cdot\partial_z^2\eta)(z')^2z''+
|\partial_z\eta|^2(z'')^2\\
&=&
\left((r'')^2+(r\theta'')^2\right)(1+(r')^2)^{-2}-2r'r''(1+(r')^2)^{-3}r^2\theta'\theta''\\
& &
+(1+(r')^2)^{-4}r^4(\theta'\theta'')^2(r')^2+\text{remainder}\\
&=&
(r\theta'')^2(1+(r')^2)^{-2}+\text{remainder}\\
&\approx& g'(t)^2/g(t)^6.
\end{eqnarray*}
Also recall that $\theta'''\approx -g''(t)/g(t)^4$.
The dominant term of $\partial_s(\kappa^2)$  is composed by $\theta''$ and $\theta'''$, more precisely, 
\begin{equation*}
\partial_s(\kappa^2)= 2(\partial_s\kappa) \kappa= 
2r\theta''(r\theta''')(1+(r')^2)^{-5/2}+\text{remainder}\approx g'(t)g''(t)/g(t)^7.
\end{equation*}

We also see that 
\begin{eqnarray*}
& &\kappa= |r\theta''|(1+(r')^2)^{-1}+\text{remainder}\approx g'(t)/g(t)^3,\\
&  &\kappa n\cdot e_\theta= r\theta''(1+(r')^2)^{-1}+\text{remainder}\approx -g'(t)/g(t)^3,\\
& &\partial_s(\kappa n)\cdot e_\theta=\partial_s^3\tilde\eta\cdot e_\theta
=
r\theta'''(1+(r')^2)^{-3/2}+\text{remainder}\approx-g''(t)/g(t)^4,\\
& &\partial_s\kappa=\frac{r\theta''(r\theta''')(1+(r')^2)^{-5/2}}{\kappa}+\text{remainder}\\
& &
\ \ \ \ \  
=
-r\theta'''(1+(r')^2)^{-3/2}+\text{remainder}
\approx g''(t)/g(t)^4
\end{eqnarray*}
near the blowup time.
Thus
\begin{equation*}
n\cdot e_\theta=\frac{\kappa n\cdot e_\theta}{\kappa}\to -1\quad (t\to 1).
\end{equation*}
By the Frenet-Serret formula,
\begin{equation*}
Tb=\partial_s n+\kappa\tau,
\end{equation*}
we see that 
\begin{equation*}
\kappa Tb\cdot e_\theta=\partial_s(\kappa n)\cdot e_\theta
-(\partial_s\kappa) n\cdot e_\theta+\kappa^2 \tau\cdot e_\theta.
\end{equation*}
Thus, by the direct calculation, we can find a cancellation on the (candidate) highest order term $\theta'''\approx \partial_s\kappa$, namely, 
\begin{equation*}
|\kappa Tb\cdot e_\theta|\ll |\partial_s\kappa|
\end{equation*}
for $t\in I$.
\end{proof}
In what follows, we use  a differential geometric idea. See Chan-Czubak-Y \cite[Section 2.5]{CCY}, more originally, see Ma-Wang \cite[(3.7)]{MW}.
They considered 2D separation phenomena using fundamental differential geometry. 
The key idea here is  ``local pressure estimate" on a normal coordinate in $\bar \theta$, $\bar r$ and $\bar z$ valuables.
Two derivatives to the scalar function $p$ on the normal coordinate is commutative, namely,
$\partial_{\bar r}\partial_{\bar \theta}p(\bar \theta, \bar r,\bar z)-\partial_{\bar \theta}\partial_{\bar r}p(\bar \theta,\bar r,\bar z)=0$ (Lie bracket).
This fundamental observation is the key to extract the local effect of the pressure.  
For any point $x\in\mathbb{R}^3$ near the arc-length trajectory $\eta^*$ is uniquely  expressed as $x=\eta^*(\bar \theta)+\bar r n(\bar\theta)+\bar z b(\bar\theta)$ with $(\bar\theta,\bar r,\bar z)\in\mathbb{R}^3$ (the meaning of the parameters $s$ and $\bar \theta$ are the same along the arc-length trajectory).
Thus we have that 
\begin{eqnarray*}
\partial_{\bar \theta}x&=&
\tau+\bar r(Tb-\kappa \tau)+\bar z \kappa n,\\
\partial_{\bar r}x&=& n,\\
\partial_{\bar z}x&=&b.
\end{eqnarray*}
This means that 
\begin{equation*}
\begin{pmatrix}
\partial_{\bar \theta}\\
\partial_{\bar  r}\\
\partial_{\bar z}
\end{pmatrix}
=
\begin{pmatrix}
1-\kappa \bar r& \bar z\kappa& \bar r T\\
0& 1& 0\\
0& 0& 1
\end{pmatrix}
\begin{pmatrix}
\tau\\
n\\
b\\
\end{pmatrix}.
\end{equation*}
\begin{remark}
For any smooth scalar function $f$,
we have 
\begin{equation*}
\partial_{\bar \theta}f(x)=\nabla f\cdot \partial_{\bar \theta}x.
\end{equation*}
$\nabla f$ itself is essentially independent of any coordinates. Thus we can regard a partial derivative as a vector.
\end{remark}
By the fundamental calculation, we have  the following inverse matrix:
\begin{equation*}
\begin{pmatrix}
\tau\\
n\\
b\\
\end{pmatrix}
=
\begin{pmatrix}
(1-\kappa \bar r)^{-1}& -\bar zT (1-\kappa \bar r)^{-1}& -\bar r T(1-\kappa \bar r)^{-1}\\
0& 1& 0\\
0& 0& 1
\end{pmatrix}
\begin{pmatrix}
\partial_{\bar \theta}\\
\partial_{\bar  r}\\
\partial_{\bar z}
\end{pmatrix}.
\end{equation*}
Therefore we have the following orthonormal moving frame:
 $\partial_{\bar r}=n$, $\partial_{\bar z}=b$ and 
$$
(1-\kappa \bar r)^{-1}\partial_{\bar \theta}-\bar z T(1-\kappa \bar r)^{-1}\partial_{\bar r}-\bar rT(1-\kappa \bar r)^{-1}\partial_{\bar z}=\tau.
$$

In order to abbreviate the complicated indexes, we re-define the absolute value of the velocity along the trajectory. 
 Let (the indexes are $x$ and $t$ respectively)
\begin{equation*}
|u|:=|u(\eta(x',t),t)|\quad\text{with}\quad
x'=\eta^{-1}(x,t)
\end{equation*}
and 
\begin{equation*}
\partial_t|u|:=\partial_{t'}|u(\eta(x',t'),t')|\bigg|_{t'=t}\quad\text{with}\quad
x'=\eta^{-1}(x,t).
\end{equation*}

\begin{lem}\label{Euler flow along the trajectory}
We see  $\nabla p\cdot \tau=\partial_t|u|$ along the trajectory.
\end{lem}
\begin{proof}
Let us define a unit tangent vector $\tilde\tau$ (in time $t'$) as follows:
\begin{equation*}
\tilde \tau_{x,t}(t'):=\frac{u}{|u|}(\eta(x',t'),t')\quad\text{with}\quad x'=\eta^{-1}(x,t).
\end{equation*}
Note that  there is a re-parametrize factor $s(t')$ such that 
\begin{equation*}
\tau(s(t'))=\tilde \tau(t').
\end{equation*}
Since $u\cdot \partial_s\tau=0$, we see that 
\begin{eqnarray*}
\partial_{t'}|u(\eta(x',t'),t')|&=&
\partial_{t'}(u(\eta(x',t'),t')\cdot \tilde\tau_{x,t}(t'))\\
&=&
\partial_{t'}(u(\eta(x',t'),t'))\cdot \tilde\tau_{x,t}(t')+
u(\eta (x',t'),t')\cdot\partial_s\tau\partial_{t'}s\\
&=&\partial_{t'}(u(\eta(x',t'),t'))\cdot \tilde\tau_{x,t}(t').
\end{eqnarray*}
By the above calculation we have 
\begin{equation*}
\nabla p\cdot\tau=\partial_{t'}(u(\eta(x',t'),t')\cdot\tau|_{t'=t}=\partial_{t'}(u(\eta(x',t'),t')\cdot\tilde\tau|_{t'=t}=
\partial_{t'}|u|\bigg|_{t'=t}.
\end{equation*}
\end{proof}

\begin{lem}
Along the arc-length trajectory, we have 
\begin{equation*}
3\kappa\partial_t|u|+\partial_s\kappa|u|^2=\partial_{\bar r}\partial_t|u|
\end{equation*}
and 
\begin{equation*}
T\kappa|u|^2=\partial_{\bar z}\partial_t|u|.
\end{equation*}
\end{lem}

\begin{proof}
By using the orthonormal moving frame, 
we have the following gradient of the pressure,
\begin{equation*}
\nabla p= (\partial_\tau p)\tau+(\partial_np)n+(\partial_bp)b.
\end{equation*}
By the unit tangent vector, we see 
\begin{equation*}
\partial_s\eta^*(s)=\partial_t\eta\partial_st=\tau
\end{equation*}
and thus
\begin{equation*}
\partial_st=|u|^{-1}.
\end{equation*}
By the unit normal vector with the curvature constant, we see
\begin{equation*}
\partial_s^2\eta^*=\partial_s(\partial_t\eta\partial_st)
=\partial_t^2\eta(\partial_st)^2+\partial_t\eta\partial_s^2t=\kappa n.
\end{equation*}
Thus we have 
\begin{eqnarray*}
-(\nabla p\cdot n)&=&
(\partial_t^2\eta\cdot n)=
\kappa|u|^2,\\
-\partial_s(\nabla p\cdot n)&=&\partial_s(\kappa(\partial_st)^{-2})
=\partial_s\kappa(\partial_s t)^{-2}-2\kappa(\partial_st)^{-3}(\partial_s^2t),\\
-\nabla p\cdot \tau&=&-|u|^3\partial_s^2 t,\\
-\nabla p\cdot b&=&0.
\end{eqnarray*}

Recall that  
\begin{equation*}
\partial_\tau=(1-\kappa\bar r)^{-1}\partial_{\bar\theta}-\bar zT(1-\kappa\bar r)^{-1}\partial_{\bar r}
-\bar rT(1-\kappa\bar r)^{-1}\partial_{\bar z}.
\end{equation*}
Along the arc-length trajectory, we have 
\begin{eqnarray*}
-\partial_{\bar r}(\nabla p\cdot \tau)
&=&-\partial_{\bar r}\partial_\tau p
\\
&=&
-\kappa\partial_{\bar \theta} p-\partial_{\bar r}\partial_{\bar \theta} p-T\partial_{\bar z} p\\
\nonumber
(\text{commute}\ \partial_{\bar r}\ \text{and}\  \partial_{\bar \theta})
&=&
-\kappa(\nabla p\cdot \tau)-\partial_{\bar\theta} (\nabla p\cdot n)-T(\nabla p\cdot b)\\
&=&
-\kappa |u|^3\partial_s^2t+\partial_s\kappa(\partial_st)^{-2}-2\kappa(\partial_st)^{-3}(\partial_s^2 t)\\
&=&
3\kappa\partial_t|u|+\partial_s\kappa|u|^2.
\end{eqnarray*}
Since $\nabla p\cdot b=\partial_{\bar z} p\equiv 0$ along the trajectory, then
\begin{equation*}
-\partial_{\bar z}(\nabla p\cdot \tau)|_{\bar r,\bar z=0}=-\partial_{\bar z}\partial_{\bar\theta} p-T\partial_{\bar r} p
=
-T(\nabla p\cdot n)=T\kappa|u|^2.
\end{equation*}

By Lemma \ref{Euler flow along the trajectory}
along the arc-length trajectory $\eta^*$, we have 
\begin{equation*}
3\kappa\partial_t|u|+\partial_s\kappa|u|^2=
-\partial_{\bar r}(\nabla p\cdot \tau)|_{\bar r,\bar z=0}
=
\partial_{\bar r}\partial_t|u|\\
\end{equation*}
and 
\begin{equation*}
T\kappa|u|^2=
-\partial_{\bar z}(\nabla p\cdot \tau)|_{\bar r,\bar z=0}
=
\partial_{\bar z}\partial_t|u|.\\
\end{equation*}

\end{proof}

By using the above lemma we can finally prove the main theorem.
Since 
\begin{equation*}
\partial_\theta=(e_\theta\cdot n)\partial_{\bar r}+(e_\theta\cdot b)\partial_{\bar z}
\end{equation*}
and the axisymmetric flow is rotation invariant,
\begin{eqnarray*}\label{rotation invariant}
0&=&\partial_\theta\partial_t|u|=(e_\theta\cdot n)\partial_{\bar r}\partial_t|u|
+(e_\theta\cdot b)\partial_{\bar z}\partial_t|u|\\
\nonumber
&=&
3(e_\theta\cdot n)\left(\kappa\partial_t|u|+\partial_s\kappa|u|^2\right)
+(e_\theta\cdot b)T\kappa|u|^2.
\end{eqnarray*}
However, 
$(e_\theta\cdot n)\partial_s\kappa|u|^2\approx -g''(t)/g(t)^2$ is the dominant term, and 
it is in contradiction to \eqref{flux condition}, since
\begin{equation*}
(e_\theta\cdot n)\kappa\partial_t|u|\approx -g'(t)^2/g(t)^3\quad\text{and}\quad
|T\kappa|\ll |\partial_s\kappa|.
\end{equation*}

\vspace{0.5cm}
\noindent
{\bf Acknowledgments.}\ 
The author would like to thank Professor Norikazu Saito for letting me know 
the book \cite{FQV}.
The author was partially supported by JST CREST.

\bibliographystyle{amsplain}

\end{document}